\documentclass{amsart}
\usepackage{bbm}
\usepackage{amsfonts}
\usepackage{mathrsfs}

\def\vep{\varepsilon}
\newtheorem{lem}{Lemma}[section]

\newtheorem{dfn}[lem]{Definition}
\newtheorem{pro}[lem]{Proposition}
\newtheorem{thm}[lem]{Theorem}
\newtheorem{cor}[lem]{Corollary}
\theoremstyle{remark}

\DeclareMathOperator\card{Card}
\numberwithin{equation}{section}

\begin{document}
\title[The multifractal spectra of planar waiting sets in beta expansions]{The multifractal spectra of planar waiting sets in beta expansions}
\author{Haibo Chen}
 \address{School of Statistics and Mathematics, Zhongnan University of Economics and Law, Wu\-han, Hubei, 430073, China}
  \email{hiboo\_chen@sohu.com}

\begin{abstract}
Let $\beta>1$ be a real number. In this paper, the Hausdorff dimension of sets consisting of pairs
of numbers with prescribed quantitative waiting time indicators in $\beta$-expansions are determined. More precisely, let $I$ be the unit interval $[0,1)$ and write $\underline{R}^\beta(x,y)$ and 
$\overline{R}^\beta(x,y)$ as the lower and upper quantitative waiting time
indicators of $y$ by $x$ in $\beta$-expansions, respectively. Define the waiting set on the plane by
\[E_\beta(a,b)=\left\{(x,y)\in I^2\colon\underline{R}^\beta(x,y)=a,\overline{R}^\beta(x,y)=b\right\}.\]
where $0\leq a\leq b\leq\infty$, then the set 
$E_\beta(a,b)$ is always of Hausdorff dimension two for any pair of numbers $a$ and $b$. In addition, some generalizations for this result are also given in the last section.
\end{abstract}

\maketitle

\section{Introduction}
The $\beta$-expansion was introduced by R\'{e}nyi (see \cite{R}) as a generalization of the ordinary $m$-adic expansion, where $m\geq2$ be an integer. To begin with, we first introduce the definition and some notations of $\beta$-expansion. Let $I=[0,1)$. For $\beta>1$, define the \emph{$\beta$-transformation} $T_\beta\colon I\to I$ by
\[T_\beta(x)=\beta x-[\beta x],\quad x\in I.\] 
Here, $[\xi]$ is the integer part of number $\xi$. Under this transformation, each $x\in I$ can be uniquely expanded into a finite or an infinite series as
\begin{align}\label{formula beta 1}
	x=\frac{\vep_1(x,\beta)}{\beta}+\frac{\vep_2(x,\beta)}{\beta^2}+\cdots+\frac{\vep_n(x,\beta)}{\beta^n}+\cdots,
\end{align}
where $\vep_n(x,\beta)=\big[\beta T_\beta^{n-1}(x)\big]$, $n\geq1$, is called the $n$-th digit of $x$ with respect to base $\beta$ and $\vep(x,\beta)=\big(\vep_1(x,\beta),\vep_2(x,\beta),\ldots\big)$ is called the \emph{digit sequence} of $x$. Sometimes, we can also identify $x$ with its digit sequence and write 
\begin{align}\label{formula beta 2}
	x=\big(\vep_1(x,\beta),\vep_2(x,\beta),\ldots\big) 
\end{align}
for simplicity. The formulas \eqref{formula beta 1} and \eqref{formula beta 2} are both called the $\beta$-expansion of $x$.

Denote by $\Sigma=\{0,1,\ldots,\lceil\beta-1\rceil\}$ the alphabet with $\lceil\beta\rceil$ elements, where $\lceil\xi\rceil$ is the least integer greater or equal to $\xi$. Then every $n$-th digit $\vep_n(x,\beta)$ of $x$ belongs to $\Sigma$. Moreover, denote by $\Sigma^n$, $n\geq1$, the set of finite sequences of length $n$ and $\Sigma^\infty$ the set of infinite sequences of which each digit is in $\Sigma$. It is worth noting that for a given sequence in $\Sigma^\infty$ it might not be the $\beta$-expansion of some $x\in I$. There are two operators in relation to $\Sigma^\infty$. One is the shift operator $\sigma$ on $\Sigma^\infty$ given by
\[\sigma(\vep_1,\vep_2,\vep_3,\ldots)=(\vep_2,\vep_3,\vep_4,\ldots)\]
for any $(\vep_1,\vep_2,\vep_3,\ldots)\in\Sigma^\infty$.
The other is the metric $d$ on $\Sigma^\infty\times\Sigma^\infty$, which is defined as
\begin{align}
	d(\vep,\eta)=\beta^{-\min\{i\geq1\colon\vep_i\neq\eta_i\}} 
\end{align}
where $\vep=(\vep_1,\vep_2,\ldots)\in\Sigma^\infty$ and $\eta=(\eta_1,\eta_2,\ldots)\in\Sigma^\infty$. Endowed with this metric, the space $\Sigma^\infty$ is compact and $\sigma$ is continuous on $\Sigma^\infty$.

Let $n\geq1$. We call a word $(\vep_1,\vep_2,\ldots,\vep_n)$ or a sequence $(\vep_1,\vep_2,\ldots,\vep_n,\ldots)$ $\beta$-admissible if there exists an $x\in I$ such that the $\beta$-expansion of $x$ begins with it. Here and in the sequel, we call a finite sequence a word and an infinite sequence a sequence for the sake of distinction. Denote by $\Sigma_\beta^n$, $n\geq1$, the set of all $\beta$-admissible words of length $n$ and $\Sigma_\beta$ the set of all admissible sequences, that is,
\[\Sigma_\beta=\{\vep(x,\beta)\colon x\in I\}.\] 
Let $S_\beta$ be the closure of $\Sigma_\beta$ under the product topology on $\Sigma^\infty$. Then $(S_\beta,\sigma|_{S_\beta})$ is a subshift of the symbolic space $(\Sigma^\infty,\sigma)$ and the two systems $(S_\beta,\sigma|_{S_\beta})$ and $(I,T_\beta)$ are metrically isomorphism. 

For an admissible word $(\vep_1,\ldots,\vep_n)$, $n\geq1$, we call 
\[I_n(\vep_1,\ldots,\vep_n)=\{x\in I\colon \vep_i(x,\beta)=\vep_i, 1\leq i\leq n\}\]
an $n$-th order \emph{basic interval}. Actually, the $n$-th order basic interval is a left-open and right-closed interval with $\vep_1/\beta+\cdots+\vep_n/\beta^n$ as its left endpoint (see Lemma 2.3 in \cite{FW}). Moreover, write $I_n(x)=I_n\big(\vep_1(x,\beta),\ldots,\vep_n(x,\beta)\big)$ for the $n$-th order basic interval containing $x$ and $|I_n(x)|$ for the length of $I_n(x)$.

Now, we would like to introduce the definition of quantitative waiting time indicator, for more one can refer to~\cite{CT,G1,WZ} and the references therein. Let $\beta>1$, $n\geq1$ and $x,y\in I$. We call
\[W_n^\beta(x,y)=\inf\left\{k\geq1\colon T_\beta^kx\in I_n(y)\right\}\]
\emph{the waiting time} of $I_n(y)$ by $x$ in $\beta$-expansion. That is, $W_n^\beta(x,y)$ is the first entrance time of the orbit of $x$ under $T_\beta$ into the basic interval $I_n(y)$. Moreover, we can also write it as
\[W_n^\beta(x,y)=\inf\left\{k\geq1\colon \big(\vep_k(x,\beta),\ldots,\vep_{k+n-1}(x,\beta)\big)=\big(\vep_1(y,\beta),\ldots,\vep_n(y,\beta)\big)\right\}.\] Upon it,
\emph{the lower and upper quantitative waiting time indicators} of $y$ by $x$ in $\beta$-expansions are defined respectively as
	\begin{align}
	    \underline{R}^\beta(x,y)=\liminf_{n\to\infty}
        \frac{\log W_n^\beta(x,y)}{n}\quad\mbox{and}\quad \overline{R}^\beta(x,y)=\limsup_{n\to\infty}
        \frac{\log W_n^\beta(x,y)}{n}.
	\end{align}
If the two indicators are common, then we denote the value by $R^\beta(x,y)$ and call it \emph{the quantitative waiting time indicator} of $y$ by $x$. In addition, if $W_n^\beta(x,y)$ is infinite for some pair of $x$ and $y$, then $\underline{R}^\beta(x,y)$ and $\overline{R}^\beta(x,y)$ are set to be equal to infinity.

From the above definitions, we can see that the quantitative waiting time indicator gives a quantitative description to the speed at which a given trajectory of one point enters into the neighborhood of another point.

In fact, the quantitative waiting time indicator can be regarded as an generalization of the quantitative recurrence time indicator, which is studied firstly by
L.~Barreira and B.~Saussol in~\cite{BS}.  For the development of quantitative recurrence time indicator and its application in the continued fraction dynamical system, L\"{u}roth dynamical system, etc, one can refer to some noted work such as \cite{FengWu}, \cite{OW}, \cite{P} and \cite{PTW}.

In recent years, the quantitative waiting time indicator has aroused the interest of many people. Compared with the quantitative recurrence time indicator, by using it we can give more precise analysis and obtain more information on the theory of recurrence. To illustrate this point, we would like to introduce a significant work, given by Marton and
Shields~\cite{MS}, which proved that for a weak Bernoulli process the quantitative waiting time indicator
equals to the entropy almost surely with respect to the product
measure. Since Bernoulli shifts are also weak Bernoulli and $\beta$-shifts is Bernoulli (see \cite{IMT}), we may obtain as a special case that
\begin{equation}\label{R(x,y)}
R^\beta(x,y)=\log\beta,\quad\mbox{a.e.}\ (x, y)\in I^2 
\end{equation}
with respect to the product measure $\mu_\beta\times\mu_\beta$. Here, $\log\beta$ is the entropy of dynamical system $(I,T_\beta)$, $\mu_\beta$ is the unique invariant measure which is equivalent to Lebesgue measure and $T_\beta$ is ergodic with respect to $\mu_\beta$ (see \cite{Pa,R}). Note that the quantitative waiting time indicator does not convergent in general ergodic processes according to a counter example provided by Shields~\cite{Shields}. But for some special situations there are still some expectable results in the topics such as the continued fraction
transformation, the interval exchange map and the irrational rotation on the unit interval, one can see \cite{DaKa,KM,KS} and the references therein for more comprehension. 

In this paper, we would like to complement \eqref{R(x,y)} by investigating the following planar sets in which the lower and upper quantitative waiting time indicators take prescribed values.  
Let $a$ and $b$ be two real numbers satisfying $0\leq a\leq b\leq\infty$. Define the planar waiting set of two parameters 
\[E_\beta(a,b)=\left\{(x,y)\in I^2\colon\underline{R}^\beta(x,y)=a,\overline{R}^\beta(x,y)=b\right\}.\]
Denote by $\dim_H$ the Hausdorff dimension of a set. Here, we recommend Falconer's book \cite{F97} for one to consult the definition and basic properties of Hausdorff dimension. Then, for the size of $E_\beta(a,b)$ we have  

\begin{thm}\label{theorem main theorem}
	Let $\beta>1$. For any $0\leq a\leq b\leq\infty$, we have $\dim_HE_\beta(a,b)=2$. 
\end{thm}

In fact, Theorem~\ref{theorem main theorem} also generalizes simultaneously the result in \cite{CT} from $m$-adic expansion (integer $m\geq2$) to $\beta$-expansion. However, we will apply different technique and approach in its proof. Particularly, let $0\leq a\leq\infty$ and define the level set
\[E_\beta(a)=\left\{(x,y)\in I^2\colon R^\beta(x,y)=a\right\}.\] 
Then Theorem~\ref{theorem main theorem} yields immediately  
\begin{cor} 
	Let $\beta>1$. For any $0\leq a\leq\infty$, we have 
	$\dim_HE_\beta(a)=2$. 
\end{cor}

The above two results indicate that the exceptional set of \eqref{R(x,y)} is of full Hausdorff dimension two and meanwhile the sets $E_\beta(a,b)$ and $E_\beta(a)$ are both of rich multifractal structures.

The outline of this paper is shown as below. In Section 2, we will collect some notations and properties about $\beta$-expansion and full basic intervals. In Section 3, we will present some dimensional results about homogeneous sets and cut sets for ready use. Section 4 is devoted to the proof of Theorem~\ref{theorem main theorem} by constructing sufficiently large Moran sets. In the last section, we will generalize Theorem~\ref{theorem main theorem} by considering the accumulation points of the sequence $\big\{\frac{1}{n}\log W_n^\beta(x,y)\big\}_{n\geq1}$ and other growth rates of the waiting time of $I_n(y)$ by $x$. 

\section{Full basic intervals}
This section is devoted to the introduction of some notations, definitions and basic properties of $\beta$-expansion and full basic intervals.

First, we introduce the definition of $\beta$-expansion of 1 which plays an important role in the description of admissible sequences. Let $\beta>1$. According to~\eqref{formula beta 1}, if the $\beta$-expansion of 1 terminates, i.e., there exists $n\geq1$ such that $\vep_n(1,\beta)\neq0$ and $\vep_i(1,\beta)=0$ for all $i\geq n+1$, then we call $\beta$ a simple Parry number and put
\[(\vep_1^*(1,\beta),\vep_2^*(1,\beta),\ldots)=(\vep_1(1,\beta),\vep_2(1,\beta),\ldots,\vep_n(1,\beta)-1)^\infty.\] 
Here, $(w)^\infty$ denotes the periodic sequence $(w,w,w,\ldots)$ when $w$ is a word. Otherwise, we write $\vep_i^*(1,\beta)=\vep_i(1,\beta)$, $i\geq1$, and use $(\vep_1^*(1,\beta),\vep_2^*(1,\beta),\ldots)$ to denote the $\beta$-expansion of 1. In both cases, the infinite $\beta$-expansion of 1 is denoted by
\[\vep^*(1,\beta)=(\vep_1^*(1,\beta),\vep_2^*(1,\beta),\ldots).\]
Let $<_{lex}$ be the lexicographical order between two sequences in $\Sigma^\infty$. We say 
\[(\vep_1,\vep_2,\vep_3,\ldots)<_{lex}(\eta_1,\eta_2,\eta_3,\ldots)\]
if and only if $\vep_1<\eta_1$ or there exists $n\geq1$ such that $\vep_i=\eta_i$ for $1\leq i<n$ but $\vep_n<\eta_n$. 

The following theorem gives a result of Parry on charactering whether a digit sequence is admissible. 

\begin{thm}[Parry \cite{Pa}]\label{theorem parry}
  \noindent	
  \begin{enumerate}
    \item A non-negative integer sequence $(\vep_1,\vep_2,\ldots)$ is $\beta$-admissible if and only if 
    \[(\vep_i,\vep_{i+1},\ldots)<_{lex}(\vep_1^*(1,\beta),\vep_2^*(1,\beta),\ldots),\quad\forall i\geq1.\]
    \item If $1<\beta_1<\beta_2$, then $\Sigma_{\beta_1}\subset\Sigma_{\beta_2}$ and $\Sigma_{\beta_1}^n\subset\Sigma_{\beta_2}^n$, $n\geq1$.
  \end{enumerate}	
\end{thm}

Moreover, the following theorem, given by R\'{e}nyi, provides a description of the cardinality of $\Sigma_\beta^n$ and indicates that the dynamical system $(I,T_\beta)$ admits $\log\beta$ as its topological entropy.

\begin{thm}[R\'{e}nyi \cite{R}]\label{theorem renyi}
	Let $\beta>1$. For any $n\geq1$, we have
	\begin{align}\label{inequality renyi}
		\beta^n\leq\card\Sigma_\beta^n\leq\frac{\beta^{n+1}}{\beta-1}, 
	\end{align}
    where $\card$ denotes the number of elements in a finite set.
\end{thm}  

Let the word $(\vep_1,\ldots,\vep_n)$ be admissible. It is obvious that the length of the $n$-th order basic interval $I_n(\vep_1,\ldots,\vep_n)$ is less than or equal to $\beta^{-n}$. If $\beta$ is an integer, then the equality always holds for any word of length $n$. Besides this good case, we might encounter another fine one, called full basic interval, which plays an important role in the present work.  

\begin{dfn}
	Let $(\vep_1,\ldots,\vep_n)\in\Sigma_\beta^n$, $n\geq1$. An $n$-th order basic interval $I_n(\vep_1,\ldots,\vep_n)$ is said to be \emph{full} if its length verifies
	\[|I_n(\vep_1,\ldots,\vep_n)|=\beta^{-n}.\]
\end{dfn}

The following lemma characterizes the full basic intervals.

\begin{lem}[See Lemma 3.1 in \cite{FW}]\label{lemma 31}
Let $\vep=(\vep_1,\ldots,\vep_n)\in\Sigma_\beta^n$, $n\geq1$. The following conditions are equivalent:
\begin{enumerate}
	\item $I_n(\vep_1,\ldots,\vep_n)$ is full;
	\item $T_\beta^nI_n(\vep_1,\ldots,\vep_n)=I$;
	\item For any $\eta=(\eta_1,\ldots,\eta_m)\in\Sigma_\beta^m$, $m\geq 1$, the concatenated word 
	\[\vep\ast\eta=(\vep_1,\ldots,\vep_n,\eta_1,\ldots,\eta_m)\]
	is admissible.
\end{enumerate}	
\end{lem}

Below are some properties of the full basic intervals. As a note, we use $w^p$ to denote a word composed by $p$ $w$'s if the integer $p\geq1$ and $w$ is a letter or a word, $w^0$ is defined as the empty word $\emptyset$ and the concatenation $w\ast\emptyset$ is defined as $w$ itself.

\begin{lem}[See Lemma 3.2 and Corollary 3.3 in \cite{FW}]\label{lemma 32}
	Let $m,n\geq1$.
	\begin{enumerate}
		\item If $I_n(\vep_1,\ldots,\vep_n)$ is full, then for any word $(\eta_1,\ldots,\eta_m)\in\Sigma_\beta^m$, we have
		\[|I_{n+m}(\vep_1,\ldots,\vep_n,\eta_1,\ldots,\eta_m)|=|I_n(\vep_1,\ldots,\vep_n)|\cdot|I_m(\eta_1,\ldots,\eta_m)|;\]
		\item Let $p\in\mathbb{N}$. Then
		\[I_{n+p}(\vep_1,\ldots,\vep_n,0^p)\ \text{is full}\Leftrightarrow I_n(\vep_1,\ldots,\vep_n)\geq\beta^{-(n+p)}.\]	
	\end{enumerate}	
\end{lem}

By Lemma~\ref{lemma 32} (1) and the definition of full basic interval, we can easily obtain  
\begin{cor}\label{corollary full word}
	Let $m,n\geq1$. If the basic intervals $I_n(\vep_1,\ldots,\vep_n)$ and $I_m(\eta_1,\ldots,\eta_m)$ are full, then $I_{n+m}(\vep_1,\ldots,\vep_n,\eta_1,\ldots,\eta_m)$ is also full.
\end{cor}

Let $\beta>1$ and denote
\[l_n(\beta)=\max\left\{k\geq0\colon\vep_{n+j}^*(1,\beta)=0,\ \text{for all}\ 1\leq j\leq k\right\},\quad n\geq1.\]
That is, $l_n(\beta)$ is the length of the longest consecutive zeros following the digit $\vep_n^*(1,\beta)$ in the $\beta$-expansion of 1. 
Then, more generally, we have

\begin{lem}[See \cite{LW}]\label{lemma LW} 
	Let $n\geq1$ and $M_n=\max_{1\leq i\leq n}\{l_i(\beta)\}$. For any admissible word $(\vep_1,\ldots,\vep_n)$, if $m\geq M_n(\beta)$, then the basic interval $I_{n+m}(\vep_1,\ldots,\vep_n,0^{m+1})$ is full.
\end{lem}

Next, we would like to introduce an important subset $A_0 $ of $(1,\infty)$ (see~\cite{LW}). As it turns out, many problems about $\beta$-expansions are easily dealt with if $\beta$ takes value in $A_0$. Then, by the technique of approximation, the corresponding problems for the general $\beta>1$ can be  completely solved at last. 

The set $A_0 $ is defined as follows:
\begin{align}\label{definition b0}
	A_0=\big\{\beta>1\colon \{l_n(\beta)\}_{n\geq0}\ \text{is bounded}\big\}.
\end{align}
It has the following properties.

\begin{lem}[See \cite{LTWW,S}]\label{lemma A zero}
	$A_0$ is uncountable and dense in $(1,\infty)$. Moreover, we have $\mathcal{L}(A_0)=0$ and $\dim_HA_0=1$, where $\mathcal{L}$ is the Lebesgue measure. 
\end{lem}

Since the sequence $\{l_n(\beta)\}_{n\geq1}$ is bounded for each $\beta$ in $A_0$, by Lemma~\ref{lemma 31}, Lemma~\ref{lemma 32} and Lemma~\ref{lemma LW}, we may obtain  

\begin{lem}\label{lemma 0 m}
	Let $\beta\in A_0$ and $(\vep_1,\ldots,\vep_n)$, $n\geq1$, be any admissible word. There exists an integer $M>0$ such that $I_{n+M}(\vep_1,\ldots,\vep_n,0^M)$ is full. Moreover, we have 
	\begin{enumerate}
		\item  the word $(\vep_1,\ldots,\vep_n,0^M)$ can be concatenated behind by any admissible word $(\eta_1,\ldots,\eta_m)$ with $m\geq1$;
		\item the length of basic interval $I_n(\vep_1,\ldots,\vep_n)$ satisfies that
		\[\beta^{-(n+M)}\leq|I_n(\vep_1,\ldots,\vep_n)|\leq\beta^{-n}.\]
	\end{enumerate}	
\end{lem}

\section{Some lemmas}
In this section, we collect some definitions and lemmas for ready use. The scope includes homogeneous Moran sets, an invariant property of Hausdorff dimension related to the subset of $\mathbb{N}$ with density zero, cut sets and the existence of sequences with prescribed growth rates.

At first, we introduce the definition of homogeneous Moran set. Let  $\{N_k\}_{k\geq1}$ be a sequence of integers and $\{c_k\}_{k\geq1}$ be a sequence of positive numbers satisfying $N_k\geq2$, $0<c_k<1$, $k\geq1$; and  $N_1c_1\leq\delta$, $N_kc_k\leq1$, $k\geq2$, $\delta>0$. Denote
\[D_0=\{\emptyset\},\ D_k=\{(i_1,\ldots,i_k)\colon 1\leq i_j\leq N_j,1\leq j\leq k\}\quad\text{and}\quad D=\bigcup_{k\geq0}D_k.\]
Suppose that $J$ is an interval of length $\delta$. A collection $\mathcal{F}=\{J_\sigma\colon \sigma\in D\}$ of subintervals of $J$ is said to have \emph{homogeneous structure} if it satisfies the following conditions:
\begin{enumerate}
	\item $J_{\emptyset}=J$;
	\item For any $\sigma\in D_{k-1}$ with $k\geq1$, $J_{\sigma\ast j}$, $1\leq j\leq N_k$, are subintervals of $J_\sigma$ and $\text{int}(J_{\sigma\ast i})\cap\text{int}(J_{\sigma\ast j})=\emptyset$ if $i\neq j$, where $\text{int}$ denotes the interior of a set;
	\item For any $\sigma\in D_{k-1}$ with $k\geq1$, we have
	\[\frac{|J_{\sigma\ast j}|}{J_\sigma}=c_j,\quad 1\leq j\leq N_k.\]
\end{enumerate} 
If the collection $\mathcal{F}$ is of homogeneous structure, then the set
\begin{align}\label{definition mj}
	\mathcal{M}=\mathcal{M}\big(J,\{N_k\}_{k\geq1},\{c_k\}_{k\geq1}\big)=\bigcap_{k\geq1}\bigcup_{\sigma\in D_k}J_\sigma
\end{align} 
is called a \emph{homogeneous Moran set} determined by $\mathcal{F}$.

\begin{lem}[See Theorem 2.1 and Corollary 2.1 in \cite{FWW}]\label{lemma FWW}
	Let $\mathcal{M}$ be the homogeneous Moran set defined in~\eqref{definition mj} and write
	\[s=\liminf_{k\to\infty}\frac{\log (N_1N_2\cdots N_k)}{-\log(c_1c_2\cdots c_{k+1}N_{k+1})}.\] 
	Then we have $\dim_H\mathcal{M}\geq s$. 
	In addition, if $\inf_{k\geq1}c_k>0$, then $\dim_H\mathcal{M}=s$.
\end{lem}

Next, we introduce a lemma about the invariance of Hausdorff dimension of a set under a mapping by deleting a set of positions of digit sequences of numbers with density zero in $\mathbb{N}$. Let $\mathbb{M}$ be a subset of $\mathbb{N}$ and write  
\[\mathbb{N}\backslash\mathbb{M}=\{n_1<n_2<\ldots\}.\] 
Let $D\subset I$ and denote by $\Sigma_D$ the set of digit sequences of numbers in $D$. That is, \[\Sigma_D=\{(\vep_1(x,\beta),\vep_2(x,\beta),\ldots)\in\Sigma_\beta\colon x\in D\}.\] 
Suppose there is a mapping
$\phi_\mathbb{M} \colon D\to I$ such that the corresponding digit sequences satisfy
\[(\vep_1(x,\beta),\vep_2(x,\beta),\ldots)\in\Sigma_D\mapsto(\vep_{n_1}(x,\beta),\vep_{n_2}(x,\beta),\ldots)\in\Sigma_\beta,\] 
If there exist such pair of set $D$ and mapping $\phi_\mathbb{M}$, then we call the mapping $\phi_\mathbb{M}$ \emph{maps well} on the set $D$.
Thus, given such set $D\subset I$ and mapping $\phi_\mathbb{M}$, we may obtain another set
\[\phi_\mathbb{M} (D)=\{\phi_\mathbb{M} (x)\colon x\in D\}.\] 
Moreover, we call the set $\mathbb{M} $ is of \emph{density zero} in $\mathbb{N}$ if
\[\lim_{n\to\infty}\frac{\card\{i\in\mathbb M\colon i\leq n\}}{n}=0.\]
The following lemma describes the relation between the sizes of $D$ and $\phi_\mathbb{M} (D)$.

\begin{lem}[See Lemma 5.1 in \cite{Chen}]\label{lemma density}
	Let $\beta\in A_0$, $D\subset I$ and $\mathbb{M} $ be of density zero in $\mathbb{N}$. If the mapping $\phi_\mathbb{M}$ maps well on $D$, then we have
	\[\dim_H\phi_\mathbb{M} (D)=\dim_HD.\]
\end{lem}

In what follows, we introduce the definition and dimensional result about cut set.

\begin{dfn}
	Let $F$ be a subset of $I^2$. For any number $y\in I$, define
	\[F(y)=\{x\in I\colon(x,y)\in F\}.\]
	We call $F(y)$ \emph{the cut set} of $F$ by number $y$.
\end{dfn}
For the Hausdorff dimensions of the sets $F$ and $F(y)$, we have the
following relation.

\begin{lem}[See Corollary 7.12 in~\cite{F97}]\label{CUT}
	Let $E\subset I$ and $F\subset I^2$. If 
	$\dim_HF(y)\geq t$ for any $y\in E$, then 
	\[\dim_HF\geq t+\dim_H E.\]
\end{lem}

At last, we present a crucial result about the existence of a sequence of integers $\{l_n\}_{n\geq1}$ which is of fine prescribed increasing behavior. 

\begin{lem}[See~\cite{FW}]\label{FW}
	For any $0\leq a\leq b\leq\infty$, there
	exists an increasing sequence of integers $\{l_n\}_{n\geq1}$ such
	that
	\begin{align}
		\liminf_{n\to\infty}\frac{\log l_n}{n}=a\quad\mbox{and} \quad \limsup_{n\to\infty}\frac{\log l_n}{n}=b.
	\end{align}
	Moreover, $\{l_n\}_{n\geq1}$ increases faster
	than the growth rate of the sequence $\{[e^{\sqrt{n}}]\}_{n\geq1}$.
\end{lem}

\section{Proof of Theorem \ref{theorem main theorem}}
In this section, we will give the proof of Theorem \ref{theorem main theorem}. To this end, we introduce below some notations and present some lemmas for later use. 

Let $\beta\in A_0$, $n\geq1$ and $m\geq M$. Here and in the sequel, $M$ is always given by Lemma~\ref{lemma 0 m}. Write
\[V_{n,m}^\beta=\left\{(\vep_1,\ldots,\vep_n,0^m)\colon (\vep_1,\ldots,\vep_n)\in\Sigma_\beta^n\right\}.\]
Note that we always have the fact $\vep_1\neq0$ by the definition of admissible sequence. Based on it, define a subset of the unit interval
\begin{align*}
	\mathcal{V}_{n,m}^\beta&=\left\{x\in I\colon \big(\vep_{(n+m)(i-1)+1}(x,\beta),\ldots,\vep_{(n+m)i})(x,\beta)\big)\in V_{n,m}^\beta,\forall i\geq1\right\}\\
	&=:\big(V_{n,m}^\beta\big)^\infty.
\end{align*}
That is, each digit sequence of number in $\mathcal{V}_{n,m}^\beta$ is the infinite concatenations of finite sequences of length $n+m$ in $\mathcal{V}_{n,m}^\beta$.
Next, put  
\begin{align}
	\mathcal{V}_m^\beta=\bigcup_{n\geq1}\mathcal{V}_{n,m}^\beta.
\end{align}
Then we have

\begin{lem}\label{lemma wv1}
	Let $\beta\in A_0$. For any $m\geq M$, we have 
	\[\dim_H\mathcal{V}_m^\beta=1.\] 
\end{lem}
\begin{proof}
	By Lemma \ref{theorem renyi} and Lemma \ref{lemma FWW}, we have
	\begin{align}\label{formula wnm}
		\dim_H\mathcal{V}_{n,m}^\beta\geq\liminf_{k\to\infty}\frac{k\log\beta^n}{-(k+1)\log\beta^{-(n+m)}-\log\beta^{n}}=\frac{n}{n+m}  
	\end{align}
	for $n\geq1$. Since $\mathcal{V}_{n,m}^\beta\subset\mathcal{V}_m^\beta$, we have 
	\[\dim_H\mathcal{V}_m^\beta\geq\frac{n}{n+m}\]
	for any $n\geq1$. It follows that
	$\dim_H\mathcal{V}_m^\beta=1$ by letting $n\to\infty$.	
\end{proof}

Let $0\leq a\leq b\leq\infty$ and $y\in\mathcal{V}_m^\beta$, where $\beta\in A_0$. Define the cut set of $E_\beta(a,b)$ by
\begin{align}
	E_{a,b}^\beta(y)=\left\{x\in I\colon\underline{R}^\beta(x,y)=a,\overline{R}^\beta(x,y)=b\right\}. 
\end{align}
Then we have
\begin{lem}\label{lemma faby1}
	Let $\beta\in A_0$ and $m\geq M$. For any $0\leq a\leq b\leq\infty$ and $y\in\mathcal{V}_m^\beta$, we have
	\[\dim_HE_{a,b}^\beta(y)=1.\]
\end{lem}
\begin{proof}
	Let $y\in\mathcal{V}_{n_0,m_0}^\beta$ for some $n_0\geq1$ and $m_0\geq M$.
	Let $N\geq1$ be an integer and fix an integer $p\geq n_0+m_0$. 
	Based on the set $\mathcal{V}_{N,p}^\beta$ and the sequence $\{l_n\}_{n\geq1}$ given in Lemma~\ref{FW}, in the following we will construct a subset, say $\mathcal{V}^\beta(\{l_n\},y)$, of $E_{a,b}^\beta(y)$.
	
	For each $x\in\mathcal{V}_{N,p}^\beta$, we first construct a number $x^*$  by induction. Write \[x^{(0)}=x=\big(\vep_1(x^{(0)},\beta),\vep_2(x^{(0)},\beta),\ldots\big).\] 
	Here, we identify $x$ with its digit sequence for convenience. Suppose we have defined
	\[x^{(j)}=\big(\vep_1(x^{(j)},\beta),\vep_2(x^{(j)},\beta),\ldots\big),\quad\text{for}\ 0\leq j\leq k,\]
	then take an integer $p\geq n_0+m_0$ and define
	\[x^{(k+1)}=\big(\vep_1(x^{(k)},\beta),\dots,\vep_{l_k}(x^{(k)},\beta), V_{0,p,k}^\beta(y),\vep_{l_k+1}(x^{(k)},\beta),\ldots\big),\]
	where
	\begin{align}\label{v0pk}
		V_{0,p,k}^\beta(y)=\big(0^p,\vep_1(y,\beta),\ldots,\vep_{k+1}(y,\beta),0^p\big).
	\end{align} 
	That is, $x^{(k+1)}$ is obtained by inserting a word $V_{0,p,k}^\beta(y)$ of length $2p+k+1$ at the position $l_k+1$ of the digit sequence of $x^{(k)}$.  
	Since the sequence	$\{l_k\}_{k\geq1}$ increases faster than the linear growth rate of the lengths of words $V_{0,p,k}^\beta(y)$, $k\geq0$, we can require $k_0$ to be large enough so
	that $V_{0,p,k}^\beta(y)$ and $V_{0,p,k+1}^\beta(y)$ can not be overlapped when $k\geq k_0$. Thus,
	the prefix word 
	\[\big(\vep_1(x^{(k-1)},\beta),\ldots \vep_{l_{k-1}}(x^{(k-1)},\beta),V_{0,p,k}^\beta(y)\big)\] 
	of $x^{(k+1)}$ is also the prefix word of $x^{(k+2)}$. This implies
	that $\{x^{(k)}\}_{k\geq 0}$ is a Cauchy sequence. Denote by \[x^*=\big(\vep_1(x^*,\beta),\vep_2(x^*,\beta),\ldots\big)\]
	the corresponding limit point of the sequence $\{x^{(k)}\}_{k\geq 0}$ and $\mathcal{V}_{N,p,n_0,m_0}^\beta(\{l_n\},y)$ the collection of all these $x^*$'s. Here and in what follows, $\mathcal{V}_{N,p,n_0,m_0}^\beta(\{l_n\},y)$ is written as $\mathcal{V}^\beta(\{l_n\},y)$ for short. That is,
	\begin{align*}
		\mathcal{V}^\beta(\{l_n\},y)=\left\{x^*\in I\colon \text{constructed from}\ x\in\mathcal{V}_{N,p}^\beta,\ y\in\mathcal{V}_{n_0,m_0}^\beta\ \text{and}\ \{l_n\}_{n\geq1}\right\}.
	\end{align*}  
	
	Note that the length of the consecutive 0's in the digit sequence of $y$ is at most $n_0+m_0-1$, which is strictly less than that of numbers in $\mathcal{V}_{N,p}^\beta$ according to the facts $p\geq n_0+m_0$ and $\vep_{(n_0+m_0)i+1}(y,\beta)\neq0$ for all $i\geq0$. Moreover, for $x^\ast\in\mathcal{V}_{N,p}^\beta(\{l_n\},y)$, we can easily check that
	\begin{align}\label{inequality lnp}
		l_{n-n_0-m_0}+p\leq W_n^\beta(x^\ast,y)\leq l_{n-1}+p
	\end{align}
	for sufficiently large $n$. In fact, the right side of inequality \eqref{inequality lnp} is obvious and the left side holds if the equation 
	\[\big(\vep_{n-(n_0+m_0-2)}(y,\beta),\ldots,\vep_{n}(y,\beta)\big)=0^{n_0+m_0-1}\]
	would be possible. In this situation, we could find a copy of $\big(\vep_1(y,\beta),\ldots,\vep_n(y,\beta)\big)$ at the position $l_{n-1-(n_0+m_0)-1}+p=l_{n-n_0-m_0}+p$ in the digit sequence of $x^*$ when $n$ is large enough. However, according to the characters of the values of digits of $x$ and $y$, we can not find this copy before this position.	
	
	Thus, by Lemma~\ref{FW}, we have
	\[\underline{R}^\beta(x^\ast,y)=\liminf_{n\to\infty}\frac{\log
		W_n^\beta(x^\ast,y)}{n}\leq\liminf_{n\to\infty}\Big(\frac{\log (l_{n-1}+p)}{n-1}\cdot\frac{n-1}{n}\Big)=a\] 
	and 
	\[\underline{R}^\beta(x^\ast,y)\geq\liminf_{n\to\infty}\Big(\frac{\log (l_{n-n_0-m_0}+p)}{n-n_0-m_0}\cdot\frac{n-n_0-m_0}{n}\Big)=a,\]
	which imply that $\underline{R}^\beta(x^\ast,y)=a$. Similarly, we may also deduce that
	$\overline{R}^\beta(x^\ast,y)=b$.
	Thus, we obtain $x^\ast\in E_{a,b}^\beta(y)$ and then the relation $\mathcal{V}^\beta(\{l_n\},y)\subset E_{a,b}^\beta(y)$. It follows that 
	\begin{equation}\label{S_N1}
	\dim_H\mathcal{V}^\beta(\{l_n\},y)\leq\dim_HE_{a,b}^\beta(y).
	\end{equation}
	
	Define a mapping $\phi_{\{l_n\}}\colon \mathcal{V}^\beta(\{l_n\},y)\to \mathcal{V}_{N,p}^\beta$ by $x^\ast\mapsto
	x$. Then $\phi_{\{l_n\}}$ maps well on $\mathcal{V}^\beta(\{l_n\},y)$. Moreover, according to Lemma~\ref{FW}, we have that
	\begin{align*}
		&\limsup_{n\to\infty}\frac{\card\left\{i\leq n\colon x_i^\ast\ \mbox{appears in some word}\ V_{0,p,k}^\beta(y)\right\}}{n}\\
		&\leq\lim_{k\to\infty}\frac{\sum_{j=0}^k(2p+j+1)}{l_k+2p+k+1}=0.
	\end{align*}
	This implies that the set of positions occupied by the sequence of words $V_{0,p,k}^\beta(y)$, $k\geq0$, in each digit sequence of number in $\mathcal{V}^\beta(\{l_n\},y)$ is of density zero in $\mathbb{N}$. Thus,
	by Lemma~\ref{lemma density} and the inequality \eqref{formula wnm}, we gain that
	\begin{equation}\label{S_N2}
	\dim_H\mathcal{V}^\beta(\{l_n\},y)=\dim_H\phi_{\{l_n\}}\big(\mathcal{V}^\beta(\{l_n\},y)\big)=\dim_H\mathcal{V}_{N,p}^\beta\geq\frac{N}{N+p}.
	\end{equation}
	
	On combining~(\ref{S_N1}) and~(\ref{S_N2}), it follows that
	\[\dim_HE_{a,b}^\beta(y)\geq\frac{N}{N+p}.\] 
	Then, by letting $N\to\infty$, we finish the proof.	
\end{proof}

\begin{cor}\label{corollary ab2}
	Let $\beta\in A_0$. For any $0\leq a\leq b\leq\infty$, we have
	\[\dim_HE_\beta(a,b)=2.\]
\end{cor}
\begin{proof}
	Let $m\geq M$. Define 
	\begin{align}
		E_\beta^m(a,b)=\left\{(x,y)\in I\times\mathcal{V}_m^\beta\colon\underline{R}^\beta(x,y)=a,\overline{R}^\beta(x,y)=b\right\} 
	\end{align}
	where $0\leq a\leq b\leq\infty$. Then the cut set of $E_\beta^m(a,b)$ by $y$ is just $E_{a,b}^\beta(y)$, which satisfies that 
	\[\dim_HE_{a,b}^\beta(y)=1 \]
	for any $y\in\mathcal{V}_m^\beta$ by Lemma~\ref{lemma faby1}. Thus, we have
	\[\dim_HE_\beta^m(a,b)\geq 1+\dim_H\mathcal{V}_m^\beta=2\]
	by Lemma \ref{CUT}.	This, together with the apparent fact
	$E_\beta^m(a,b)\subset E_\beta(a,b)$, finishes the proof. 
\end{proof}

Now, we are ready to give the proof of Theorem \ref{theorem main theorem}, in which the Schmeling's approximation method is applied, one can refer to \cite{TW} for more details.  

\begin{proof}[Proof of Theorem \ref{theorem main theorem}]

	Let $(\vep_1^*(1,\beta),\vep_2^*(1,\beta),\ldots)$ be the infinite $\beta$-expansion of 1 given in the beginning of Section 2. Assume $\beta_N=\beta_N(\beta)$ is the unique positive root of the equation
	\begin{align}\label{definition beta m}
		1=\frac{\vep_1^*(1,\beta)}{\beta_N^1}+\frac{\vep_2^*(1,\beta)}{\beta_N^2}+\cdots+\frac{\vep_N^*(1,\beta)}{\beta_N^N},\quad\text{where}\  \vep_N^*(1,\beta)\geq1.
	\end{align}
    All these subscripts of the roots $\beta_N$'s form a subset of $\mathbb{N}$ or a sequence, named $\{N_i\}_{i\geq1}$.
    
	Write $A_1(\beta)$ as the collection of all the roots of \eqref{definition beta m}, i.e.,
	\[A_1(\beta)=\big\{\beta_{N_i}\colon\text{the root of \eqref{definition beta m} for each}\ i\geq1\big\}.\]
	Since the digit sequence of expansion of 1 under base $\beta_{N_i}$ is the ${N_i}$-periodic sequence \[\big(\vep_1^\ast(1,\beta),\ldots,\vep_{N_i-1}^\ast(1,\beta),\vep_{N_i}^\ast(1,\beta)-1\big)^\infty\]  
	for every $i\geq1$, we have $A_1(\beta)\setminus\{1\}\subset A_0$. 
	Besides, we also have $\beta_{N_i}<\beta$ and $\beta_{N_i}$ increases to $\beta$ as $i\to\infty$. Moreover, by (2) in Theorem \ref{theorem parry}  the relation 
	\[\Sigma_{\beta_{N_i}}\subset\Sigma_{\beta_{N_j}}\subset\Sigma_\beta\]
	holds if $N_i<N_j$. 
	
	Then, denote
	\[A_1=\bigcup_{\beta>1}\big(A_1(\beta)\setminus\{1\}\big).\]
	It is clear that $A_1\subset A_0$. Also, $A_1$ is dense in the parameter space $(1,\infty)$ since $\beta_{N_i}(\beta)\to\beta$ as $i\to\infty$ for any $\beta\in(1,\infty)$.
	
	Define a mapping $\pi_\beta\colon S_\beta\to I$ which satisfies
	\begin{align}\label{definition pi}
		\pi_\beta(\vep)=\sum_{i=1}^{\infty}\frac{\vep_i}{\beta^i},\quad\vep=(\vep_1,\vep_2,\ldots)\in S_\beta.
	\end{align}
	It is easy to see that the mapping $\pi_\beta$ is a one-to-one mapping for all but countable many digit sequences. Moreover, $\pi_\beta$ is continuous on $S_\beta$ and satisfies 
	\[\pi_\beta\circ\sigma=T_\beta\circ\pi_\beta. \]
		
	For each $i$ such that $\beta_{N_i}>1$, write $D_{\beta,\beta_{N_i}}=\pi_\beta(\Sigma_{\beta_{N_i}})$ and define the mapping $g\colon D_{\beta,\beta_{N_i}}^2\to I^2$ which satisfies
	\[g(x,y)=\big(\pi_{\beta_{N_i}}(\vep(x,\beta)),\pi_{\beta_{N_i}}(\vep(y,\beta))\big),\quad x,y\in D_{\beta,\beta_{N_i}}.\]
	Then we have the following properties:
	\begin{enumerate}
		\item $\vep(g_1(x,y),\beta_{N_i})=\vep(x,\beta)$ and $\vep(g_2(x,y),\beta_{N_i})=\vep(y,\beta)$, where $g_1(x,y)=\pi_{\beta_{N_i}}(\vep(x,\beta))$ and $g_2(x,y)=\pi_{\beta_{N_i}}(\vep(y,\beta))$;
		
		\item $g\big(E_\beta(a,b)\cap D_{\beta,\beta_{N_i}}^2\big)=E_{\beta_{N_i}}(a.b)$;
		
		\item the function $g$ is $(\log\beta_{N_i}/\log\beta)$-Lipschitz on $D_{\beta,\beta_{N_i}}^2$.
	\end{enumerate}
	Here, the first two properties are obvious and the last property is followed by the elementary inequality
	\[\frac{x_1^t+x_2^t}{2}\leq\Big(\frac{x_1+x_2}{2}\Big)^t,\quad \text{where}\ x_1,x_2>0\ \text{and}\ 0<t<1,\]	
	and Theorem 3.1 in \cite{BL} which indicates that the functions $g_1$ and $g_2$ are both $(\log\beta_{N_i}/\log\beta)$-Lipschitz on $D_{\beta,\beta_{N_i}}$. 
	
	Thus, according to the above properties, we have 
	\[\dim_HE_\beta(a,b)\geq\dim_H\big(E_\beta(a,b)\cap D_{\beta,\beta_{N_i}}^2\big)\geq\frac{\log\beta_{N_i}}{\log\beta}\dim_HE_{\beta_{N_i}}(a,b).\]
	By Corollary \ref{corollary ab2} and the fact $\beta_{N_i}\in A_1\subset A_0$, it yields that \[\dim_HE_\beta(a,b)\geq\frac{\log\beta_{N_i}}{\log\beta}\times2.\]
	Let $i\to\infty$, then $\beta_{N_i}\to\beta$. It follows that $\dim_HE_\beta(a,b)=2$. The proof is ended now.
\end{proof}

\section{Two generalizations}

In this section, we will give two generalizations to Theorem \ref{theorem main theorem} from different directions. One is to consider the accumulation points of the sequence of numbers $\big\{\frac{1}{n}\log W_n^\beta(x,y)\big\}_{n\geq1}$, the other is to consider some general growth speeds of the waiting times of $I_n(y)$ by $x$.

Let $x,y\in I$. Denote by $A\big(\frac{1}{n}\log W_n^\beta(x,y)\big)$
the set of all accumulation points of the sequence
$\big\{\frac{1}{n}\log W_n^\beta(x,y)\big\}_{n\geq1}$. Let $J$ be an
interval in $(0,\infty)$. Define the planar set
\[E^\beta(J)=\Big\{(x,y)\in I^2\colon A\Big(\frac{\log W_n(x,y)}{n}\Big)=J\Big\}.\]
Then we have

\begin{thm}
	Let $\beta>1$. If the interval $J$ is closed, then $\dim_HE^\beta(J)=2$.
\end{thm}
\begin{proof}
    The proof is similar to that of Theorem~\ref{theorem main theorem}, we will only give the outline in the following. Firstly, define the cut set of $E^\beta(J)$ as
    \[E_J^\beta(y)=\left\{x\in I\colon (x,y)\in E^\beta(J)\right\},\quad\text{where}\ y\in\mathcal{V}_m^\beta,\ m\geq M.\] 
    Then prove that
    \begin{align}
    	\dim_HE_J^\beta(y)=1,\quad\text{for any}\ \beta\in A_0.
    \end{align}
   Secondly, we may deduce that $\dim_HE^\beta(J)=2$ for any $\beta\in A_0$ by Lemma \ref{CUT}. At last, by using the Schmeling's approximation method, we can obtain the desired result.
   
   In the above processes, we only need to make an explanation to the first step while the proofs of the other two steps are natural and routine as before. Let $J=[a,b]$. Choose a sequence $\{t_n\}_{n\geq1}$ contained in $J$
   such that $\{t_n,n\geq1\}$ is dense in $J$ and
   $|t_{n+1}-t_n|\leq a/n$. Define
   \begin{align}\label{formula rn}
       r_n=\big[e^{nt_n+\sqrt{n}}\big] 
   \end{align} 
   and construct, based on the set $\mathcal{V}_{N,p}$, a set $\mathcal{V}^\beta(\{r_n\},y)$ which is the collection of limit points $x^\ast$'s of the sequence $\{x^{(n)}\}_{n\geq1}$ given by
   \begin{align}\label{equality xn}
   	   x^{(n+1)}=\big(\vep_1(x^{(n)},\beta),\dots,\vep_{r_{n}}(x^{(n)},\beta), V_{0,p,n}^\beta(y),\vep_{r_{n}+1}(x^{(n)},\beta),\ldots\big),
   \end{align} 
   where $V_{0,p,n}^\beta(y)$ is defined as \eqref{v0pk}. It is easy to deduce from \eqref{formula rn} that \[\lim_{n\to\infty}\frac{r_n}{n^2}=\infty\quad\text{and}\quad  r_{n+1}-r_n>2n\] 
   for sufficiently large $n$.    
   It leads to the conclusion that the set of positions occupied by the sequence $\big\{V_{0,p,n}^\beta(y)\big\}_{n\geq1}$ in the digit sequence of $x^\ast$ is of density zero in $\mathbb{N}$. Thus, we have
   \begin{align}\label{accumulation 1}
       \dim_H\mathcal{V}^\beta(\{r_n\},y)=\dim_H\mathcal{V}_{N,p}^\beta\geq\frac{N}{N+p}
   \end{align}
   by Lemma \ref{lemma density} and \eqref{formula wnm}. 
   
   Moreover, by the definition \eqref{formula rn} of $r_n$ we can also deduce that 
   \[\lim_{n\to\infty}\Big(\frac{1}{n}\log (r_{n+M(n)}+p)-t_n\Big)=0\] 
   if there exist two numbers $M_1$ and $M_2$ such that $M_1\leq M(n)\leq M_2$, which yields that
   $A\big(\frac{1}{n}\log (r_{n+M(n)}+p)\big)=J$ since $\{t_n,n\geq1\}$ is dense in $J$. On the other hand, we can see from \eqref{equality xn} that \[r_{n-n_0-m_0}+p\leq W_n^\beta(x^\ast,y)\leq r_{n-1}+p\]   
   for $n$ large enough. Thus, we have
   \[A\Big(\frac{1}{n}\log W_n^\beta(x^\ast,y)\Big)=J.\] 
   It follows that $\mathcal{V}^\beta(\{r_n\},y)\subset E_J^\beta(y)$ and then
   \begin{align}\label{accumulation 2}
   	    \dim_HE_J^\beta(y)\geq\dim_H\mathcal{V}^\beta(\{r_n\},y).
   \end{align}

   On combining \eqref{accumulation 1} and \eqref{accumulation 2}, we have that $\dim_HE_J^\beta(y)\geq N/(N+p)$. Then, the proof is finished by letting $N\to\infty$.   
\end{proof}

Next, we consider other general speeds of approaches of the waiting times of basic interval $I_n(y)$ by $x$. For this, define the \emph{lower and upper $\varphi$-quantitative waiting time indicators} of $y$ by $x$ in $\beta$-expansion respectively as
\begin{align}
	\underline{R}_\varphi^\beta(x,y)=\liminf_{n\to\infty}
	\frac{\log W_n^\beta(x,y)}{\varphi(n)}\quad\text{and}\quad
	\overline{R}_\varphi^\beta(x,y)=\limsup_{n\to\infty}
	\frac{\log W_n^\beta(x,y)}{\varphi(n)},
\end{align}
where $\varphi$ is a positive function defined on $\mathbb{N}$. Then, define the planar level set 
\[E_\beta^\varphi(a,b)=\left\{(x,y)\in I^2\colon \underline{R}_\varphi^\beta(x,y)=
a,\overline{R}_\varphi^\beta(x,y)=b\right\},\quad\text{where}\ 0\leq a\leq b\leq\infty.\]
For the size of $E_\beta^\varphi(a,b)$, we have
\begin{pro}\label{thoerem varphi 0}
	Let $\beta>1$. If the function $\varphi$ satisfies the condition
	\begin{align}\label{varphi n}
		\lim_{n\to\infty}n\big(\varphi(n+1)-\varphi(n)\big)=\infty,
	\end{align}
	then for any $0\leq a\leq b\leq\infty$ we have   \[\dim_HE_\beta^\varphi(a,b)=2.\]
\end{pro}
\begin{proof}
	The technique of this proof is similar to that of Theorem \ref{theorem main theorem} as well. So, in the following we would like to only give the main steps. One can see more details, especially for the construction and properties of the following sequence $\{u_n\}_{n\geq1}$, in the proof of Theorem 4.2 in \cite{CT}.
	
	First, fix $\beta\in A_0$. Note that by the condition~\eqref{varphi n}, we can find a sequence $\{u_n\}_{n\geq1}$ satisfying
	\begin{align}
		\liminf_{n\to\infty}
		\frac{\log u_n}{\varphi(n)}=a\quad\text{and}\quad
		\limsup_{n\to\infty}\frac{\log u_n}{\varphi(n)}=b.
	\end{align}	
	for each given pair of $a$ and $b$. Next, based on the sequence $\{u_n\}_{n\geq1}$, similar to the relation \eqref{S_N1}, we can construct a subset $\mathcal{V}^\beta (\{u_n\},y)$ of the unit interval which satisfies $\mathcal{V}^\beta (\{u_n\},y)\subset E_{a,b}^{\beta,\varphi}(y)$, where $y\in\mathcal{V}_m^\beta$ and $E_{a,b}^{\beta,\varphi}(y)$ is the cut set of $E_\beta^\varphi(a,b)$. It follows that
	\begin{equation}\label{equation ln1}
	  \dim_HE_{a,b}^{\beta,\varphi}(y)\geq\dim_H\mathcal{V}^\beta (\{u_n\},y).
	\end{equation}	
	
	Moreover, we can even choose the sequence $\{u_n\}_{n\geq1}$ satisfying
	\begin{align} 
		\lim_{n\to\infty}\frac{u_n}{n^2}=\infty\ \ \mbox{and}\ \ u_{n+1}\geq
		u_n+2n
	\end{align}
	for sufficiently large $n$. This condition, together with Lemma \ref{lemma density}, enables us to show that 
	\begin{align}\label{equation ln2}
		\dim_H\mathcal{V}^\beta (\{u_n\},y)=\dim_H\mathcal{V}_{N,p}^\beta\geq\frac{N}{N+p}.
	\end{align} 
	On combing \eqref{equation ln1} and \eqref{equation ln2}, it yields that
	$\dim_HE_{a,b}^{\beta,\varphi}(y)=1$ by letting $N\to\infty$. 
	
	Besides, similar to the proof of Corollary \ref{corollary ab2}, for any $\beta\in A_0$ and $0\leq a\leq b\leq\infty$ we can obtain that $\dim_HE_\beta^\varphi(a,b)=2$ by Lemma~\ref{CUT}.
	
	At last, by the step-by-step application of the approximation method as being used in the proof of Theorem \ref{theorem main theorem}, we can achieve the result $\dim_HE_\beta^\varphi(a,b)=2$.
\end{proof}

Note that we can substitute the condition \eqref{varphi n} in Theorem~\ref{thoerem varphi 0} to the equality
$\lim_{n\to\infty}\log n/\varphi(n)=0$, which can be derived from \eqref{varphi n} by the elementary Stolz-Ces\`{a}ro theorem, and the remainder of the proof can be dealt with in the same manner. 

Consider two examples: $\varphi_1(n)=n^s$ where $s>0$ and $\varphi_2(n)=t^n$ where $t>1$. For the level sets
\[E_\beta^{\varphi_i}(a,b)=\left\{(x,y)\in I^2\colon \underline{R}_{\varphi_i}^\beta(x,y)=
a,\overline{R}_{\varphi_i}^\beta(x,y)=b\right\},\]
where $0\leq a\leq b\leq\infty$ and $i=1,2$, by Theorem~\ref{thoerem varphi 0} we obtain immediately that  
\begin{cor}
	Let $\beta>1$. For any $0\leq a\leq b\leq\infty$, we have \[\dim_HE_\beta^{\varphi_1}(a,b)=\dim_HE_\beta^{\varphi_2}(a,b)=2.\]
\end{cor}

\subsection*{Acknowledgment}
This work was finished when the author visited the Laboratoire d'Analyse et de Math\'{e}matiques Appliqu\'{e}es, Universit\'{e} Paris-Est Cr\'{e}teil Val de Marne, France. Many thanks for the great help provided by the laboratory and his collaborator.

\end{document}